\documentclass{amsart}
\usepackage{amsmath,amssymb}
\usepackage{amsthm}

\theoremstyle{plain}
\newtheorem{theorem}{Theorem}[section]
\newtheorem{proposition}[theorem]{Proposition}
\newtheorem{lemma}[theorem]{Lemma}
\newtheorem{cor}[theorem]{Corollary}

\theoremstyle{remark}
\newtheorem{rem}[theorem]{Remark}
\newtheorem{ex}[theorem]{Example}
\newtheorem{definition}[theorem]{Definition}

\newcommand{\N}{{\mathbb N}}
\newcommand{\Z}{{\mathbb Z}}
\newcommand{\Q}{{\mathbb Q}}
\newcommand{\C}{{\mathbb C}}
\newcommand{\R}{{\mathbb R}}
\newcommand{\K}{{\mathbb K}}

\numberwithin{equation}{section}

\begin{document}

\title[Number systems over orders]{Number systems over orders}

\author[Attila~Peth\H{o} and J\"org Thuswaldner]{Attila~Peth\H{o} and J\"org~Thuswaldner}
\address{Department of Computer Science, University of Debrecen,
H-4010 Debrecen, P.O. Box 12, HUNGARY and\\
University of Ostrava, Faculty of Science,
Dvo\v{r}\'{a}kova 7, 70103 Ostrava, CZECH REPUBLIC}
\email{Petho.Attila@inf.unideb.hu}
\address{Chair of Mathematics and Statistics,
University of Leoben, Franz-Josef-Strasse 18, A-8700 Leoben,
AUSTRIA}
\email{Joerg.Thuswaldner@unileoben.ac.at}
\thanks{Research of A.P. is supported in part by the OTKA grants NK104208, NK115479 and the FWF grant P27050. Research of J.T. is supported by the FWF grants P27050 and P29910.}
\date{\today}

\begin{abstract}
Let $\mathbb{K}$ be a number field of degree $k$ and let $\mathcal{O}$ be an order in $\mathbb{K}$. A \emph{generalized number system over $\mathcal{O}$} (GNS for short) is a pair $(p,\mathcal{D})$ where $p \in \mathcal{O}[x]$ is monic and $\mathcal{D}\subset\mathcal{O}$ is a complete residue system modulo $p(0)$ containing $0$. If each $a \in \mathcal{O}[x]$ admits a representation of the form
$a \equiv \sum_{j =0}^{\ell-1} d_j x^j \pmod{p}$
with $\ell\in\mathbb{N}$ and $d_0,\ldots, d_{\ell-1}\in\mathcal{D}$ then the GNS $(p,\mathcal{D})$ is said to have the \emph{finiteness property}.
To a given fundamental domain $\mathcal{F}$ of the action of $\mathbb{Z}^k$ on $\mathbb{R}^k$ we associate a class $\mathcal{G}_\mathcal{F} := \{ (p, D_\mathcal{F}) \;:\; p \in \mathcal{O}[x]  \}$ of GNS whose digit sets $D_\mathcal{F}$ are defined in terms of $\mathcal{F}$ in a natural way. We are able to prove general results on the finiteness property of GNS in $\mathcal{G}_\mathcal{F}$ by giving an abstract version of the well-known ``dominant condition'' on the absolute coefficient $p(0)$ of $p$. In particular, depending on mild conditions on the topology of $\mathcal{F}$ we characterize the finiteness property of $(p(x\pm m), D_\mathcal{F})$ for fixed $p$ and large $m\in\N$. Using our new theory, we are able to give general results on the connection between power integral bases of number fields and GNS.
\end{abstract}

\date{\today}

\maketitle

\section{Introduction}
In the present paper we introduce a general notion of number system defined over 
orders of number fields. This  generalizes the well-known {\em number systems} and {\em canonical number systems} in number fields to relative extensions and allows for the definition of ``classes'' of digit sets by the use of lattice tilings. It fits in the general framework of {\it digit systems over commutative rings} defined by Scheicher {\it et al.}~\cite{SSTW}.

Before the beginning of the 1990s canonical number systems have been defined as number systems that allow to represent elements of orders (and, in particular, rings of integers) in number fields. After the pioneering work of Knuth~\cite{Knuth:60} and Penney~\cite{Penney:65}, special classes of canonical number systems have been studied by K\'atai and Szab\'o~\cite{KS75}, K\'atai and Kov\'acs~\cite{KK1,KK2}, and Gilbert~\cite{G81}, while elements of a general theory are due to Kov\'acs \cite{Kovacs} as well as Kov\'acs and Peth\H{o}~\cite{KovacsPetho,KovacsPetho1992}. In 1991 Peth\H{o}~\cite{pethoe1991:polynomial_transformation_and} gave a more general definition of canonical number systems that reads as follows. Let $p\in\mathbb{Z}[x]$ be a monic polynomial and let $\mathcal{D}$ be a complete residue system of $\mathbb{Z}$ modulo $p(0)$ containing $0$. The pair $(p,\mathcal{D})$ was called a {\it number system} if each $a\in\mathbb{Z}[x]$ allows a representation of  the form
\begin{equation}\label{eq:zcns}
a \equiv d_0 + d_1x + \cdots + d_{\ell-1}x^{\ell-1} \pmod{p} \qquad (d_0,\ldots, d_{\ell-1} \in \mathcal{D}).
\end{equation}
If such a representation exists it is unique if we forbid  ``leading zeros'', {\it i.e.}, if we demand $d_{\ell-1}\not=0$ for $a\not\equiv 0 \pmod{p}$ and take the empty expansion for $a\equiv 0 \pmod{p}$ (note that the fact that $0\in\mathcal{D}$ is important for the unicity of the representation). It can be determined algorithmically by the so-called ``backward division mapping'' (see {\it e.g.}~\cite[Section~3]{ABBPT_I} or~\cite[Lemma~2.5]{SSTW}). Choosing the digit set $\mathcal{D}=\{0,1,\ldots,|p(0)|-1\}$, the pair $(p,\mathcal{D})$ was called a {\it canonical number system}, CNS for short. An overview about the early theory of number systems can be found for instance in Akiyama {\it et al.}~\cite{ABBPT_I} and Brunotte, Huszti, and Peth\H{o}~\cite{BHP}.

Let $p\in \mathbb{Z}[x]$ and let $\mathcal{D}$ be a complete residue system modulo $p(0)$. With the development of the theory of radix representations it became necessary to notationally distinguish an arbitrary pair $(p,\mathcal{D})$ from a particular pair $(p,\mathcal{D})$ for which each $a\in\mathbb{Z}[x]$ admits a representation of the form \eqref{eq:zcns}. Nowadays in the literature an arbitrary pair $(p,\mathcal{D})$ is called {\it number system} (or {\it canonical number system} if $\mathcal{D}=\{0,1,\ldots,|p(0)|-1\}$), while the fact that each $a\in\mathbb{Z}[x]$ admits a representation of the form \eqref{eq:zcns} is distinguished by the suffix {\em with finiteness property}. Although there exist many partial results on the characterization of number systems with finiteness property with special emphasis on canonical number systems (see for instance~\cite{AkiPet,AR,Brunotte02,Brunotte12,Burcsi-Kovacs:08,Kovacs,ST,W}), a complete description of this property seems to be out of reach (although there are fairly complete results for finite field analogs which can be found {\it e.g.} in~\cite{BBST,EHM,LP}).

If $(p,\mathcal{D})$ is a number system and $a\in\mathbb{Z}[x]$ admits a representation of the form \eqref{eq:zcns}, we call $\ell$ the \emph{length} of the representation of $a$ in this number system (for  $a\equiv 0 \pmod{p}$ this length is zero by definition).

In the present paper we generalize the CNS concept in two ways. Firstly, instead of looking at polynomials in $\mathbb{Z}[x]$ we consider polynomials with coefficients in some order $\mathcal{O}$ of a given number field of degree $k$, and secondly, we consider the sets of digits in a more general but uniform way (see Definition~\ref{def:GNS}). Indeed, for each fundamental domain $\mathcal{F}$ of the action of $\mathbb{Z}^k$ on $\mathbb{R}^k$ we define a class of number systems $(p,D_\mathcal{F})$ where $\mathcal{F}$ associates a digit set $D_\mathcal{F}$ with each polynomial $p\in \mathcal{O}[x]$ in a natural way. Thus for each fundamental domain $\mathcal{F}$ we can define a class $\mathcal{G}_\mathcal{F} := \{ (p, D_\mathcal{F}) \;:\; p \in \mathcal{O}[x]  \}$ of number systems whose properties will be studied.

Our main objective will be the investigation of the finiteness property (see Definition~\ref{def:fin}) for these number systems. For a given pair $(p,\mathcal{D})$ this property can be checked algorithmically ({\it cf.}~Theorem~\ref{non-ECNS1}). This makes it possible to prove a strong bound for the length of the representations (given in Theorem~\ref{thm:length}), provided it exists.

The ``dominant condition'', a condition for the finiteness property of $(p,\mathcal{D})$ that involves the largeness of the absolute coefficient of $p$, has been studied for canonical number systems in several versions for instance in Kov\'acs~\cite[Section~3]{Kovacs}, Akiyama and Peth\H{o}~\cite[Theorem 2]{AkiPet}, Scheicher and Thuswaldner~\cite[Theorem 5.8]{ST}, and Peth\H{o} and Varga~\cite[Lemma 7.3]{PV}. The main difficulty of the generalization of the dominant condition is due to the fact that in $\mathcal{O}$ we do not have a natural ordering, hence, we cannot adapt the methods that were used in the case of integer polynomials. However, by exploiting tiling properties of the fundamental domain $\mathcal{F}$ we are able to overcome this difficulty, and provide a general criterion for the finiteness property (see Theorem~\ref{th:1}) that is in the spirit of the dominant condition and can be used in the proofs of our main results. In particular, using this criterion, depending on natural properties of $\mathcal{F}$ we are able to show that $(p(x+m), D_\mathcal{F})$ enjoys a finiteness property for each given $p$ provided that $m$ (or $|m|$) is large enough. This forms a generalization of an analogous result of Kov\'acs~\cite{Kovacs} to this general setting (see Theorem~\ref{th:newKovacs} and its consequences).  In Theorem~\ref{non-ECNSmain} we give a converse of this result by showing that $(p(x-m), D_\mathcal{F})$ doesn't enjoy the finiteness property for large $m$ if $\mathcal{F}$ has certain properties.

If $p\in\Z[x]$ is irreducible then $\mathbb{Z}[x]/(p)$ is isomorphic to $\mathbb{Z}[\alpha]$ for any root $\alpha$ of $p$. Thus in this case the finiteness property of $(p,\mathcal{D})$ is easily seen to be equivalent to the fact that each $\gamma\in\mathbb{Z}[\alpha]$ admits a unique expansion of the form
\begin{equation}\label{eq:cnsirr}
\gamma=d_0+d_1\alpha +\cdots+ d_{\ell-1}\alpha^{\ell-1}
\end{equation}
with analogous conditions on $d_0,\ldots, d_{\ell-1}\in\mathcal{D}$ as in \eqref{eq:zcns}. In this case we sometimes write $(\alpha,\mathcal{D})$ instead of $(p,\mathcal{D})$, see Section~\ref{sec:6} (note that $|p(0)|=|N_{\mathbb{Q}(\alpha)/\mathbb{Q}}(\alpha)|$). This relates number systems to the problem of \emph{power integral bases} of orders. Recall that the order $\mathcal{O}$ has a power integral basis, if there exists $\alpha\in \mathcal{O}$ such that each $\gamma \in \mathcal{O}$ can be written uniquely in the form
$$
\gamma = g_0+g_1\alpha + \dots + g_{k-1} \alpha^{k-1}
$$
with $g_0,\dots,g_{k-1} \in \Z$. In this case $\mathcal{O}$ is called {\it monogenic}. The definitions of number system with finiteness property \eqref{eq:cnsirr} and power integral bases seem similar and indeed there is a strong relation between them. Kov\'acs~\cite[Section~3]{Kovacs} proved that the order $\mathcal{O}$ has a power integral basis if and only if it contains $\alpha$ such that $(\alpha,\{0,\ldots,|N_{\mathbb{Q}(\alpha)/\mathbb{Q}}(\alpha)|-1\})$ is a CNS with finiteness property. A deep result of Gy\H{o}ry~\cite{Gyory} states that, up to translation by integers, $\mathcal{O}$ admits finitely many power integral bases and they are effectively computable. Combining this result of Gy\H{o}ry with the above mentioned theorem of Kov\'acs~\cite[Section~3]{Kovacs}, Kov\'acs and Peth\H{o}~\cite{KovacsPetho}
proved that if $1,\alpha,\dots,\alpha^{k-1}$ is a power integral basis then, up to finitely many possible exceptions, $(\alpha-m, \mathcal{N}_0(\alpha-m)), m\in \Z$ is a CNS with finiteness property if and only if $m>M(\alpha)$, where $M(\alpha)$ denotes a constant. A good overview over this circle of ideas is provided in the book of Evertse and Gy\H{o}ry~\cite{Evertse_Gyory}.

Using Theorem \ref{th:newKovacs} we generalize the results of Kov\'acs \cite{Kovacs} and of Kov\'acs and Peth\H{o}~\cite{KovacsPetho} to number systems over orders in algebraic number fields, see especially Corollary~\ref{cor:newKovacs}. Our result is not only more general as the earlier ones, but sheds fresh light to the classical case of number systems over $\Z$ too. It turns out (see Theorem~\ref{thm:newKovacsPetho}) that under general conditions in orders of algebraic number fields the power integral bases and the bases of number systems with finiteness condition up to finitely many, effectively computable exceptions coincide. Choosing for example the symmetric digit set, the conditions of Theorem~\ref{thm:newKovacsPetho} are satisfied and, hence, power integral bases and number systems coincide up to finitely many exceptions. This means that CNS are quite exceptional among number systems.

\section{Number systems over orders of number fields}

In this section we define number systems over orders and study some of their basic properties. This new notion of number system generalizes the well-known {\it canonical number systems} defined by Peth\H{o}~\cite{pethoe1991:polynomial_transformation_and} that we mentioned in the introduction. Before we give the exact definition, we introduce some notation.

Let $\mathbb{K}$ be a number field of degree $k$. The field $\K$ has $k$ isomorphisms to $\C$, whose images are called its (algebraic) conjugates and are denoted by $\K^{(1)}, \dots, \K^{(k)}$. We denote by $\alpha^{(j)}$ the image of $\alpha\in \K$ in $\K^{(j)}$, $j=1,\dots,k$. Let $\mathcal{O}$ be an order in $\mathbb{K}$, {\it i.e.}, a ring which is a $\mathbb{Z}$-module of full rank in $\mathbb{K}$. For $a(x) = \sum_{l=0}^{n} a_l x^l\in \mathcal{O}[x]$ the quantity $H(a) = \max\{|a_l^{(j)}|,\; l=0,\dots,n, \; j =1,\dots,k \}$ is called the {\it height} of $a$.

\begin{definition}[Generalized number system]\label{def:GNS}
Let $\mathbb{K}$ be a number field of degree $k$ and let $\mathcal{O}$ be an order in $\mathbb{K}$. A \emph{generalized number system over $\mathcal{O}$} (GNS for short) is a pair $(p,\mathcal{D})$, where $p \in \mathcal{O}[x]$ is monic and $\mathcal{D}\subset\mathcal{O}$ is a complete residue system in $\mathcal{O}$ modulo $p(0)$ containing $0$. The polynomial $p$ is called \emph{basis} of this number system, $\mathcal{D}$ is called its set of \emph{digits}.
\end{definition}

\begin{rem}
Note that GNS fit in the more general framework of {\it digit systems over a commutative ring} defined in Scheicher {\it et al.}~\cite{SSTW}. Indeed, \cite[Example~6.6]{SSTW} provides an example of a digit system over a commutative ring that corresponds to the case $\mathcal{O}=\Z[i]$ of our definition and uses rational integers as digits. Our more specialized setting allows us to prove results that are not valid for arbitrary commutative rings.
\end{rem}

Let 
\begin{equation}\label{eq:basisOmega}
\omega_1=1,\omega_2,\ldots,\omega_k
\end{equation}
be a $\mathbb{Z}$-basis of $\mathcal{O}$ and arrange this basis in a vector 
\begin{equation}\label{eq:BasisVector}
\boldsymbol{\omega}=(\omega_1,\omega_2,\ldots,\omega_k). 
\end{equation}
Let $\mathcal{F}$ be a bounded fundamental domain for the action of $\mathbb{Z}^k$ on $\mathbb{R}^k$, {\it i.e.}, a set that satisfies $\mathbb{R}^k=\mathcal{F}+\mathbb{Z}^k$ without overlaps, and assume that $\mathbf{0}\in \mathcal{F}$. Such a fundamental domain defines a set of digits in a natural way. Indeed, for $\vartheta\in\mathcal{O}$ define
\begin{equation}\label{eq:DF}
D_{\mathcal{F},\vartheta} = (\vartheta\cdot(\mathcal{F}\cdot\boldsymbol{\omega})) \cap \mathcal{O},
\end{equation}
where $\vartheta\cdot (\mathcal{F} \cdot {\bf \omega}) = \Big\{  \vartheta \sum_{j=1}^k f_j \omega_j \; :\; (f_1,\dots,f_k) \in \mathcal{F}\Big\}$.
Note that $D_{\mathcal{F},\vartheta}$ depends on the vector $\boldsymbol{\omega}$, {\it i.e.}, on the basis \eqref{eq:basisOmega}.

\begin{lemma}
For $\vartheta\in\mathcal{O}$ the set $D_{\mathcal{F},\vartheta}$ is a complete residue system modulo $\vartheta$.
\end{lemma}

\begin{proof}
  Let $\beta\in \mathcal{O}$. Then, because $\frac{\beta}{\vartheta} \in \K$, and the entries of $\boldsymbol{\omega}$ form a $\Q$-basis of $\K$, there exists $\mathbf{b} \in \Q^k$ with $\frac{\beta}{\vartheta} = \mathbf{b}\cdot \boldsymbol{\omega}$. Since $\mathcal{F}$ is a fundamental domain for the action of $\mathbb{Z}^k$ on $\mathbb{R}^k$ there is a unique vector $\mathbf{m}\in \mathbb{Z}^k$ satisfying $\mathbf{b} \in \mathcal{F} + \mathbf{m}$. Let $\mu=\mathbf{m}\cdot \boldsymbol{\omega}$, then $\mu\in\mathcal{O}$. Setting $\nu=\beta-\mu \vartheta$ we have $\nu\in\mathcal{O}$ and $\nu = \vartheta\cdot((\mathbf{b}-\mathbf{m})\cdot \boldsymbol{\omega})$ with $\mathbf{b}-\mathbf{m}\in\mathcal{F}$, hence, $\nu \in D_{\mathcal{F},\vartheta}$. Thus for each $\beta \in \mathcal{O}$ there is $\nu \in D_{\mathcal{F},\vartheta}$ such that $\nu \equiv \beta \pmod \vartheta$.

Let now $\tau_1,\tau_2\in D_{\mathcal{F},\vartheta}$ be given in a way that $\tau_1 \equiv \tau_2 \pmod \vartheta$. Then there is $\mathbf{r}_\ell\in\mathcal{F}$ such that $\frac{\tau_\ell}{\vartheta} = \mathbf{r}_\ell\cdot \boldsymbol{\omega}$ for $\ell\in\{1,2\}$.
As $\frac{\tau_1-\tau_2}{\vartheta} \in \mathcal{O}$, the difference $\mathbf{r}_1-\mathbf{r}_2$ is in $\Z^k$. Because $\mathcal{F}$ is a fundamental domain for the action of $\mathbb{Z}^k$ on $\mathbb{R}^k$ this is only possible if $\mathbf{r}_1-\mathbf{r}_2=\mathbf{0}$.
Thus $\tau_1=\tau_2$.
\end{proof}

We now show that each digit set $\mathcal{D}$ is of the form \eqref{eq:DF} for a suitable choice of the fundamental domain $\mathcal{F}$.

\begin{lemma}\label{lem:domaindigit}
Let $(p,\mathcal{D})$ be a GNS over $\mathcal{O}$. Then there is a  bounded fundamental domain $\mathcal{F}$ for the action of $\mathbb{Z}^k$ on $\mathbb{R}^k$ such that $\mathcal{D}=D_{\mathcal{F},p(0)}$.
\end{lemma}

\begin{proof}
Let $\iota: \mathbb{K} \to \mathbb{R}^k$ be the embedding defined by $\mathbf{r}\cdot \boldsymbol{\omega} \mapsto \mathbf{r}$. Then there is a unique matrix $P$ satisfying $\iota(p(0) \vartheta)=P\iota(\vartheta)$ for all $\vartheta \in \mathbb K$. For each $d=\sum d_j \omega_j \in \mathcal{D}$ define a cube $C_d= \prod [d_j-\frac12,d_j+\frac12)$. Then set $\mathcal F = P^{-1}\bigcup_{d\in \mathcal{D}} C_d$. It is easily verified that this is a fundamental domain satisfiying $\mathcal{D}=D_{\mathcal{F},p(0)}$.
\end{proof}

If the polynomial $p$ is clear from the context we will use the abbreviation 
\[
D_{\mathcal{F},p(0)}=D_{\mathcal{F}}.
\]
We call a fundamental domain $\mathcal{F}$ satisfying the claim of Lemma~\ref{lem:domaindigit} a \emph{fundamental domain associated with $(p,\mathcal{D})$}. In view of this lemma we may assume in the sequel that a number system $(p,\mathcal{D})$ has an associated fundamental domain $\mathcal{F}$. On the other hand, a fixed fundamental domain $\mathcal{F}$ defines a whole class of GNS, namely,
\[
\mathcal{G}_\mathcal{F} := \{ (p, D_\mathcal{F}) \;:\; p \in \mathcal{O}[x]  \}.
\]

\begin{ex}\label{ex:ex}
We consider some special choices of $\mathcal{F}$ corresponding to families $\mathcal{G}_\mathcal{F}$ studied in the literature.
\begin{description}
\item[\it Classical CNS] Canonical number systems as defined in the introduction form a special case of GNS. Indeed let $\mathbb{K}=\mathbb{Q}$ and $\mathcal{O}=\mathbb{Z}$. Then $k=1$, since $\mathbb{Q}$ has degree $1$ over $\mathbb{Q}$. Now we choose $\mathcal{F}=[0,1)$ which obviously is a fundamental domain of $\mathbb{Z}$ acting on $\mathbb{R}$. We look at the class
$
\mathcal{G}_\mathcal{F} := \{ (p, D_\mathcal{F}) \;:\; p \in \mathbb{Z}[x]  \}.
$
Since for  $\vartheta\in\mathbb{N}$ with $\vartheta\ge 2$ we have
\[
D_{\mathcal{F},\vartheta} = \left\{\tau\in \mathbb{Z} \; : \; \frac{\tau}{\vartheta} = r, r \in [0,1) \right\}=\{0,\ldots, \vartheta-1\},
\]
which is the digit set of a canonical number system, we see that the class $\mathcal{G}_\mathcal{F}$ coincides with the set of canonical number systems in this case.

If, however, $\vartheta\in\mathbb{Z}$ with $\vartheta \le -2$ then
\[
D_{\mathcal{F},\vartheta} = \left\{\tau\in \mathbb{Z} \; : \; \frac{\tau}{\vartheta} = r, r \in [0,1) \right\}=\{\vartheta+1, \ldots, 0,\} = -\{0,\ldots, |\vartheta|-1\}.
\]

\item[\it Symmetric CNS] Symmetric CNS are defined in the same way as CNS, apart from the digit set. Indeed, $(p,\mathcal{D})$ is a symmetric CNS if $p\in\mathbb{Z}[x]$ and $\mathcal{D}=\big[-\frac{|p(0)}2,\frac{|p(0)|}2\big) \cap \mathbb{Z}$. These number systems were studied for instance by Akiyama and Scheicher~\cite[Section~2]{Akiyama-Scheicher:07} and Brunotte~\cite{Brunotte09} (see also K\'atai~\cite[Theorem~7]{Katai:95} for a slightly shifted version and Scheicher {\it et al.}~\cite[Definition~5.5]{SSTW} for a more general setting). They are easily seen to be equal to the class $\mathcal{G}_\mathcal{F} := \{ (p, D_\mathcal{F}) \;:\; p \in \mathbb{Z}[x]  \}$ with $\mathcal{F}=[-\frac12,\frac12)$ of GNS.

\item[\it The sail] Let $\mathbb{K}=\Q(\sqrt{-D})$ with $D\in\{1,2,3,7,11\}$ be an Euclidean quadratic field with ring of integers ({\it i.e.}, maximal order) $\mathcal{O}$ and set
\[
    \omega = \left\{ \begin{array}{ll}
                       \sqrt{-D}, & \mbox{if}\; -D\equiv 2,3 \pmod{4},  \\
                       \frac{1+\sqrt{-D}}{2}, &  \mbox{otherwise}.
                     \end{array} \right.
    \]
Defining
\[
\mathcal{F}_{\omega}= \Big\{ (r_1,r_2)\in \R^2 \;:\;  |r_1+r_2\omega| <1, |r_1-1+r_2\omega|\ge 1, \; -\frac12\le r_2< \frac12
  \Big\}
\]
(this set looks a bit like a sail) one immediately checks that in  Peth\H{o} and Varga~\cite{PV} the class of GNS $\mathcal{G}_\mathcal{F} := \{ (p, D_\mathcal{F}) \;:\; p \in \mathcal{O}[x]  \}$ with $\mathcal{F} = \mathcal{F}_{\omega}$ is investigated. Using the modified fundamental domain
\[
\mathcal{F}_{\omega}= \Big\{ (r_1,r_2)\in \R^2 \;:\;  ||(r_1,r_2)||_2 <1, ||(r_1-1,r_2)||_2\ge 1, \; -\frac12\le r_2< \frac12
  \Big\}
\]
even yields a class of GNS for any imaginary quadratic number field.
\item[\it The square] As a last example we mention the number systems over $\mathbb{Z}[i]$ studied by Jacob and Reveilles~\cite{JR} and Brunotte {\it et al.}~\cite{BKT}. They correspond to the class $\mathcal{G}_\mathcal{F} := \{ (p, D_\mathcal{F}) \;:\; p \in \mathcal{O}[x]  \}$ of GNS with $K=\mathbb{Q}(i)$, $\mathcal{O}=\mathbb{Z}[i]$, and $\mathcal{F}=[0,1)^2$.
\end{description}
\end{ex}

A fundamental domain $\mathcal{F}$ induces by definition a tiling of $\mathbb{R}^k$ by $\mathbb{Z}^k$-translates which in turn induces the following neighbor relation on $\mathbb{Z}^k$. We call $\mathbf{z}'\in \mathbb{Z}^k$ a \emph{neighbor} of $\mathbf{z}\in \mathbb{Z}^k$ if $\mathcal{F}+\mathbf{z}$ ``touches'' $\mathcal{F}+\mathbf{z}'$, {\it i.e.}, if $(\overline{\mathcal{F}} + \mathbf{z}) \cap (\overline{\mathcal{F}}+\mathbf{z}') \not= \emptyset$. Note that also $\mathbf{z}$ itself is a neighbor of $\mathbf{z}$.  Let 
\begin{equation}\label{eq:neighbours}
\mathcal{N} = \{ \mathbf{z} \in \mathbb{Z}^k  \;:\;   \overline{\mathcal{F}} \cap (\overline{\mathcal{F}}+\mathbf{z}) \not= \emptyset\}
\end{equation}
be the set of neighbors of $\mathbf 0$.
We need the following easy result.

\begin{lemma}\label{lem:neighborbasis}
The set of neighbors $\mathcal{N}$ of $\mathcal{F}$ contains a basis of the lattice $\mathbb{Z}^k $.
\end{lemma}

\begin{proof}
Assume that this is wrong and let $\sim$ be the transitive hull of the neighbor relation on $\mathbb{Z}^k$. It is easy to see that this is an equivalence relation. By assumption there is $\mathbf{z}\in\mathbb{Z}^k$ such that $\mathbf{0} \not\sim \mathbf{z}$ and, hence, there are at least two equivalence classes of $\sim$. Let $C$ be one of them. Then $C$ and $\mathbb{Z}^k \setminus C$ are contained in pairwise disjoint unions of equivalence classes. Since $\mathcal{F}$ is bounded and the union $\bigcup_{\mathbf{z}\in\mathbb{Z}^k}(\mathcal{F}+\mathbf{z})$ is locally finite this implies that the nonempty sets
$
A= \bigcup_{\mathbf{z}\in C}(\mathcal{F}+\mathbf{z})
$ and
$B=
\bigcup_{\mathbf{z}\in \mathbb{Z}^k \setminus C}
(\mathcal{F}+\mathbf{z})
$
satisfy $\overline{A}\cap \overline{B} = \emptyset$ and, by the tiling property, $A\cup B=\mathbb{R}^k$. This is absurd because it would imply that $\mathbb{R}^k$ is disconnected.
\end{proof}

Let $(p,\mathcal{D})$ be a GNS and $a \in \mathcal{O}[x]$. We say that $a$ admits a {\em finite digit representation} if there exist  $\ell\in\mathbb{N}$  and  $d_0,\ldots, d_{\ell-1}\in\mathcal{D}$ such that
\[
a \equiv \sum_{j =0}^{\ell-1} d_j x^j \pmod{p}.
\]
If $d_{\ell-1}\not= 0$ or $\ell=0$ (which results in the empty sum) then $\ell$ is called the {\em length} of the representation of $a$. It will be denoted by $L(a)$. A ``good'' number system admits finite digit representations of all elements. We give a precise definition for GNS having this property.

\begin{definition}[Finiteness property]\label{def:fin}
Let $(p,\mathcal{D})$ be a GNS and set
\[
R(p,\mathcal{D}) := \bigg\{
a \in \mathcal{O}[x] \;: \;
a \equiv  \sum_{j =0}^{\ell-1} d_j x^j \pmod{p}
\quad \hbox{with } \ell\in\mathbb{N} \hbox{ and } d_0,\ldots, d_{\ell-1}\in\mathcal{D}
\bigg\}.
\]
The GNS $(p,\mathcal{D})$ is said to have the \emph{finiteness property} if $R(p,\mathcal{D})=\mathcal{O}[x]$.
\end{definition}

As in the case $\mathcal{O}=\mathbb{Z}$ with canonical digit set (see \cite[Theorem~6.1~(i)]{pethoe1991:polynomial_transformation_and} and \cite[Theorem~3]{KovacsPetho}), also in our general setting the finiteness property of $(p,\mathcal{D})$ implies expansiveness of the basis $p$ in the sense stated in the next result. (Note that its proof is also reminiscent of the proof of Vince~\cite[Proposition~4]{Vince}; this is a related result in the context of self-replicating tilings.)

\begin{proposition} \label{Prop1}
Let $(p,\mathcal{D})$ be a GNS with finiteness property. Then all roots of each conjugate polynomial $p^{(j)}(x)$,  $j\in\{1,\dots,k\}$, lie outside the closed unit disk.
\end{proposition}

\begin{proof}
Assume that there exists $|\alpha| < 1$ which is a root of $p^{(j)}(x)$ for some $j\in \{1,\dots,k\}$. By the finiteness property of $(p,\mathcal{D})$ for each  $m\in \N\subset\mathcal{O}[x]$ there exists $a(x)\in \mathcal{D}[x]$ such that $m \equiv a(x) \pmod{p}$. Thus, taking conjugates implies that $m \equiv a^{(j)}(x) \pmod{p^{(j)}}$, where $a^{(j)}(x)\in \mathcal{D}^{(j)}[x]$ with $\mathcal{D}^{(j)}= \{\beta^{(j)}\,:\, \beta \in \mathcal{D}\}$. Inserting $\alpha$ in the last congruence we get $m = a^{(j)}(\alpha)$. As $\mathcal{D}^{(j)}$ is a finite set and $|\alpha|<1$ the set of the numbers $|m|=|a^{(j)}(\alpha)|$ is bounded, which is a contradiction to $m$ being an arbitrary rational integer. Since $j$ was arbitrary, $|\alpha| \ge 1$ has to hold for all roots of $p^{(j)}(x)$ with $j\in\{1,\dots,k\}$.

 Assume now that $|\alpha|=1$ holds for a root $\alpha$ of $p^{(j)}(x)$ for some $j\in \{1,\dots,k\}$. The element $\alpha$ is an algebraic integer, and a root of the polynomial $\prod_{i=1}^{k}p^{(i)}(x)\in \Z[x]$. By \cite[Satz 3]{HK} (see the proof of \cite[Theorem~3]{KovacsPetho} and \cite[Theorem~6.1~(i)]{pethoe1991:polynomial_transformation_and}, too), $\alpha$ is a root of unity of some order, say $s$. Let $\tilde g(x) = \gcd (p^{(j)}(x), x^s-1)$. Plainly there is $c\in \mathcal{O}^{(j)}$ such that $g(x)=c \tilde g(x) \in \mathcal{O}^{(j)}[x]$, and $\deg g \ge 1$ because $\alpha$ is a root of $g$. Moreover, there are $c_1,c_2\in \mathcal{O}^{(j)}$ such that $c_1(x^s-1) = g_1(x)g(x)$ and $ c_2 p^{(j)}(x)=g_2(x)g(x)$ hold with $g_1,g_2\in \mathcal{O}^{(j)}[x]$. Then
$$
g_2(x)c_1(x^s-1) = g_1(x)g_2(x)g(x) \equiv 0 \pmod{p^{(j)}},
$$
thus $c_1g_2(x)(x^{h s}-1) \equiv 0 \pmod{p^{(j)}}$, and equivalently $$c_1g_2(x) \equiv c_1g_2(x) x^{hs} \pmod{p^{(j)}}$$ is true for all $h\ge 1$.

Let $h_2(x) \in \mathcal{O}[x]$ be the inverse image of $c_1 g_2(x)$ with respect to the isomorphism $\K \to \K^{(j)}$. As $(p,\mathcal{D})$ is a GNS with finiteness property, there exists a unique $a(x) \in \mathcal{D}[x]$ such that $h_2(x) \equiv a(x) \pmod{p}$. Thus $a^{(j)}(x)$ is the unique element in $\mathcal{D}^{(j)}[x]$ with $c_1g_2(x) \equiv a^{(j)}(x) \pmod{p^{(j)}}$. Let $t$ denote the degree of $a(x)$. Choosing $h$ so that $hs>t$ we obtain
\[
c_1g_2(x) \equiv a^{(j)}(x)  \equiv x^{hs}a^{(j)}(x)  \pmod{p^{(j)}},
\]
which contradicts the uniqueness of $a^{(j)}(x)$.
\end{proof}

Adapting the proof of Akiyama and Rao~\cite[Proposition~2.3]{AR} or Peth\H{o}~\cite[Theorem~1]{Petho:06} to orders one can prove the following algorithmic criterion for checking the finiteness property of a given GNS $(p, \mathcal{D})$. Note that there exist only finitely many $a\in \mathcal{O}[x]$ of bounded degree and bounded height.

\begin{theorem} \label{non-ECNS1}
Let $\mathbb{K}$ be a number field of degree $k$ and let $\mathcal{O}$ be an order in $\mathbb{K}$.
Let $(p,\mathcal{D})$ be a GNS over $\mathcal{O}$. There exists an explicitly computable constant $C$ depending only on $p$ and $\mathcal{D}$ such that $(p,\mathcal{D})$ is a GNS with finiteness property if and only if the polynomial $\prod_{i=1}^{k} p^{(i)}(x)$ is expansive and
\[
\{
a \in \mathcal{O}[x] \;:\; \deg a < \deg p \hbox{ and } H(a) \le C
\} \subset R(p,\mathcal{D}).
\]
\end{theorem}

\begin{proof}
The necessity assertion is an immediate consequence of Proposition~\ref{Prop1}, hence, we have to prove only the sufficiency assertion.

Let $p\in \mathcal{O}[x]$ be given in a way that $\prod_{i=1}^{k} p^{(i)}(x)$ is expansive. Denote by $\mathcal{O}_{\deg p}[x]$ the set of elements of $\mathcal{O}[x]$ of degree less than $\deg p$. For any $b'\in \mathcal{O}[x]$ there exists a unique $b\in \mathcal{O}_{\deg p}[x]$ such that $b\equiv b' \pmod{p}$. Thus it is sufficient to show that $\mathcal{O}_{\deg p}[x] \subset R(p,\mathcal{D})$.

Let $T_p \; :\; \mathcal{O}_{\deg p}[x] \to \mathcal{O}_{\deg p}[x]$ be the \emph{backward division mapping}, which is defined as
$$
T_p(b)(x) = \frac{b(x) - qp(x)- d_0}{x},
$$
where $d_0\in \mathcal{D}$ is the unique digit with $d_0 \equiv b(0) \pmod{p(0)}$ and $q=\frac{b(0)-d_0}{p(0)}$. Iterating  $T_p$ for $h$-times we obtain $d_0,\dots, d_{h-1} \in \mathcal{D}$, and $r\in \mathcal{O}[x]$ such that
\begin{equation}\label{eq:bRepH}
b(x) = \sum_{j=0}^{h-1} d_j x^j + x^h T_p^{h}(b)(x) + r(x)p(x).
\end{equation}
Clearly, $b\in R(p,\mathcal{D})$ if and only if $T_p^h(b)\in R(p,\mathcal{D})$ for all $h\ge 0$. Taking conjugates in \eqref{eq:bRepH} we get
\begin{equation}\label{eq:bRepH2}
b^{(i)}(x) = \sum_{j=0}^{h-1} d^{(i)}_j x^j + x^h T_p^{h}(b)^{(i)}(x) + r^{(i)}(x)p^{(i)}(x),\quad i=1,\dots,k.
\end{equation}
In the remaining part of the proof, for the sake of simplicity we assume that $p$ is irreducible in $\mathbb{K}[x]$. The general case can be treated by adapting the proof of Akiyama and Rao~\cite[Proposition~2.3]{AR} or of Peth\H{o} \cite[Theorem~1]{Petho:06}.

Denote by $\alpha_{i\ell}$ the roots of $p^{(i)}(x)$, $i=1,\dots,k$; $\ell=1,\dots,\deg p$. By assumption, their modulus is larger than one. Inserting $\alpha_{i\ell}$ into \eqref{eq:bRepH2} we obtain
$$
T_p^{h}(b)^{(i)}(\alpha_{i\ell}) = \frac{b^{(i)}(\alpha_{i\ell})}{\alpha_{i\ell}^h} - \sum_{j=0}^{h-1} d^{(i)}_j \alpha_{i\ell}^{j-h}, \quad i=1,\dots,k;\; \ell=1,\dots,\deg p.
$$
Taking absolute values, choosing $h$ large enough, and using the fact that $|\alpha_{i\ell}|>1$ we obtain
\begin{equation}\label{AdditiveCF}
|T_p^{h}(b)^{(i)}(\alpha_{i\ell})| \le \frac{\max\{|d^{(i)}|\; : \; d \in \mathcal{D}, i=1,\dots,k\}}{1-|\alpha_{i\ell}|^{-1}} +1, \;\;  (1\le i \le k;\;1 \le \ell\le \deg p).
\end{equation}
As the polynomial $T_p^h(b)$ is of degree at most $\deg p-1$, we may write it in the form $T_p^h(b)(x) = \sum_{j=0}^{\deg p-1} d_{hj} x^j$. Then $T_p^h(b)^{(i)}(x) = \sum_{j=0}^{\deg p-1} d_{hj}^{(i)} x^j, i=1,\dots,k$. Considering \eqref{AdditiveCF} as a system of inequalities in the unknowns $d_{hj}^{(i)}$ we obtain
$$
|d_{hj}^{(i)}| < C_{ij},\;  i=1,\dots,k;\, j=0,\dots,\deg p-1.
$$
Indeed, this is true because \eqref{AdditiveCF} says that all the Galois conjugates of the element $T_p^{h}(b)(\alpha_{11}) \in \mathcal{O}[\alpha_{11}]$ are bounded by the explicit bounds given in \eqref{AdditiveCF}. This is true only for finitely many elements of the order $\mathcal{O}[\alpha_{11}]$ in the field $\mathbb{K}(\alpha_{11})$ (which has degree $k\deg p$ over $\mathbb{Q}$), and these elements can be explicitly computed. 
Choosing $C = \max\{C_{ij}, i=1,\dots,k;\, j=0,\dots,\deg p-1\}$, we obtain $H(T_p^h(b))\le C$.

Thus for any $b \in \mathcal{O}[x]$ (which may be assumed w.l.o.g.\ to satisfy $\deg b < \deg p$) there is $a\in \mathcal{O}[x]$ with $\deg a < \deg p$, and $H(a)\le C$, namely $a = T_p^h(b)$, such that $a\in R(p,\mathcal{D})$ if and only if $b\in R(p,\mathcal{D})$. Thus $\{
a \in \mathcal{O}[x] \;:\; \deg a < \deg p \hbox{ and } H(a) \le C
\} \subset R(p,\mathcal{D})$ implies that $(p,\mathcal{D})$ is a GNS with finiteness property.
\end{proof}

Theorem~\ref{non-ECNS1} implies that the GNS property is algorithmically decidable. One has to apply the backward division mapping defined above to all polynomials satisfying $\deg a < \deg p$ and $H(a) \le C$ iteratively.  During the iteration process one always works with polynomials satisfying these inequalities. More on algorithms for checking the finiteness property of GNS can be found in a more general context in Scheicher {\it et al.}~\cite[Section~6]{SSTW}.

The proof of Theorem~\ref{non-ECNS1} makes it possible to prove a precise bound for the length of a representation in a GNS $(p,\mathcal{D})$ with finiteness property. This is in complete agreement with an analogous result of Kov\'acs and Peth\H{o}~\cite{KovacsPetho1992} for the case $\mathcal{O}=\Z$.

\begin{theorem} \label{thm:length}
Let $\mathbb{K}$ be a number field of degree $k$ and let $\mathcal{O}$ be an order in $\mathbb{K}$.
Let $(p,\mathcal{D})$ be a GNS over $\mathcal{O}$. Denote by $\alpha_{i\ell}$ the zeros of $p^{(i)}(x)$,  $i=1,\dots,k,\; \ell=1,\dots,\deg p$. If $p$ is irreducible and $(p,\mathcal{D})$ satisfies the finiteness property then there exists an explicitly computable constant $C$ depending only on $p$ and $\mathcal{D}$ such that
$$
L(a) \le \max \left\{\frac{\log|a^{(i)}(\alpha_{i\ell})|}{\log|\alpha_{i\ell}|}\; : \; i=1,\dots,k,\, \ell=1,\dots,\deg p\right\} + C
$$
holds for all $a \in \mathcal{O}[x].$
\end{theorem}

\begin{proof}
We will use the notation of the proof of Theorem~\ref{non-ECNS1}. Let $a\in\mathcal{O}[x]$ and choose $h$ in a way that
$$
\left | \frac{a^{(i)}(\alpha_{i\ell})}{\alpha^h_{i\ell}} \right | \le 1
$$
holds for all $i=1,\dots,k$, $\ell=1,\dots,\deg p$. Since by Proposition~\ref{Prop1} we have $|\alpha_{i\ell}|>1$ for all $i=1,\dots,k$, $\ell=1,\dots,\deg p$, the choice
$$
h = \max \left\{\frac{\log|a^{(i)}(\alpha_{i\ell})|}{\log|\alpha_{i\ell}|}\; : \; i=1,\dots,k,\, \ell=1,\dots,\deg p\right\}
$$
is suitable to achieve the required inequality. Using this choice of $h$, in the same way as in the proof of Theorem~\ref{non-ECNS1} (see in particular \eqref{AdditiveCF}) we obtain
$$
|T_p^{h}(a)^{(i)}(\alpha_{i\ell})| \le \frac{\max\{|d^{(i)}|\; : \; d \in \mathcal{D},\, i=1,\dots,k\}}{1-|\alpha_{i\ell}|^{-1}} + 1, \;  i=1,\dots,k,\,\ell=1,\dots,\deg p.
$$

There exist only finitely many $b\in \mathcal{O}[x]$ such that
$$
|b^{(i)}(\alpha_{i\ell})| \le \frac{\max\{|d^{(i)}|\; : \; d \in \mathcal{D},\, i=1,\dots,k\}}{1-|\alpha_{i\ell}|^{-1}} + 1, \;  i=1,\dots,k,\, \ell=1,\dots,\deg p,
$$
and all of them have finite representation in $(p,\mathcal{D})$ because $(p,\mathcal{D})$ is by assumption a GNS with finiteness property. Letting $C$ be the maximal length of the representations of such polynomials we get $L(a) = h+ L(b) \le h+C$, and the theorem is proved.
\end{proof}

With a little more effort one could replace irreducibility of $p$ by separability of $p$ in the statement of Theorem~\ref{thm:length}. 

\section{A general criterion for the finiteness property}

There exist some easy-to-state sufficient conditions for the finiteness property of a CNS $(p,\mathcal{D})$ in the case $\mathcal{O}=\Z$, see {\it e.g.} Kov\'acs \cite[Section~3]{Kovacs}, Akiyama and Peth\H{o} \cite[Theorem~2]{AkiPet}, Scheicher and Thuswaldner \cite[Theorem~5.8]{ST}, or Peth\H{o} and Varga \cite[Lemma~7.3]{PV}. In each of these results $|p(0)|$ dominates over the other coefficients of $p$. In general, $\mathcal{O}$ does not have a natural ordering. However, inclusion properties of some sets can be used to express dominance of coefficients in $\mathcal{O}$. This is the message of the Theorem~\ref{th:1}, which will be proved in this section. Before we state it, we introduce some notation.

For $p(x)=x^n+p_{n-1}x^{n-1}+\cdots+p_0 \in \mathcal{O}[x]$ let $(p,\mathcal{D})$ be a GNS and let $\mathcal{F}$ be an associated fundamental domain. Let the basis $\omega_1=1,\omega_2,\ldots, \omega_k$ be given as in \eqref{eq:basisOmega}, set $\boldsymbol{\omega}=(\omega_1,\ldots, \omega_k)$ as in \eqref{eq:BasisVector}, and recall the definition of the set $\mathcal{N}$ of neighbors of $\mathbf{0}$ in \eqref{eq:neighbours}.
Set  (letting $p_n=1$)
\begin{equation}\label{eq:deltaz}
\Delta = \mathcal{N}\cdot \boldsymbol{\omega} 
\quad
\hbox{and}
\quad
Z=\bigg\{
\sum_{j=1}^n \delta_j p_j \;:\; \delta_j\in \Delta
\bigg\},
\end{equation}
and note that, since $\mathcal{F}$ is bounded, these sets are finite.

\begin{theorem}\label{th:1}
Let $p(x)=x^n+p_{n-1}x^{n-1}+\cdots+p_0\in \mathcal{O}[x]$ and $(p,\mathcal{D})$ be a GNS. Let $\mathcal{F}$ be an associated fundamental domain and define $\Delta$ and $Z$ as in \eqref{eq:deltaz}. Assume that the following conditions hold (setting $p_n=1$):
\begin{itemize}
\item[(i)] $Z + \mathcal{D} \subset \mathcal{D} + p_0\Delta$, \vskip 2mm
\item[(ii)] $Z \subset \mathcal{D} \cup (\mathcal{D} - p_0)$, \vskip 2mm
\item[(iii)] $\left\{ \sum_{j\in J} p_j \;:\; J\subseteq\{1,\ldots, n\}  \right\} \subseteq \mathcal{D}$.
\end{itemize}
Then $(p,\mathcal{D})$ has the finiteness property.
\end{theorem}

We note that in the statement of Theorem~\ref{th:1} the set $\Delta$ from \eqref{eq:deltaz} can be replaced by an arbitrary set that contains $0$ and generates $\mathcal{O}$ as a semigroup and the result still remains true by the same proof. Since we only need Theorem~\ref{th:1} for our particular choice of $\Delta$ we stated the theorem for this particular case.

To prove Theorem~\ref{th:1} we need the following auxiliary result.

\begin{lemma}\label{lem:inductive}
The GNS $(p,\mathcal{D})$ has the finiteness property if and only if
for each $a\in R(p,\mathcal{D})$ and each $\alpha \in \Delta$ we have $a+\alpha \in R(p,\mathcal{D})$.
\end{lemma}

\begin{proof}
The necessity of the condition is obvious, so we are left with proving its sufficiency. Assume that for each $a\in R(p,\mathcal{D})$ and each $\alpha \in \Delta$ we have $a+\alpha \in R(p,\mathcal{D})$. By Lemma~\ref{lem:neighborbasis}, the set $\Delta$ generates $\mathcal{O}$ as a semigroup. Thus in order to prove the finiteness property it is sufficient to show that $a \in R(p,\mathcal{D})$ implies that 
\begin{equation}\label{eq:assIV}
a+\alpha x^m \in R(p,\mathcal{D}) \hbox{ for each } \alpha \in \Delta \hbox{ and each }m\ge 0. 
\end{equation}
The case  $m=0$ is true by assumption. Now choose $m\ge 1$. Let $a\in R(p,\mathcal{D})$. To conclude the proof we have to show that $a+\alpha x^{m} \in R(p,\mathcal{D})$ holds for each $\alpha \in \Delta$. We may write $a(x) \equiv \sum_{j=0}^{\ell-1}d_j x^j\pmod{p}$ with $d_0,\ldots, d_{\ell-1}\in\mathcal{D}$. Then
\begin{equation}\label{eq:Aontheway}
a(x)+ \alpha x^{m} \equiv \sum_{j=0}^{m-1} d_jx^j + x^{m}(\tilde a(x) + \alpha) \pmod{p}
\end{equation}
holds with $\tilde a(x)=\sum_{j=m}^{\ell-1} d_jx^{j-m} \in R(p,\mathcal{D})$. Since $\alpha \in \Delta$, and \eqref{eq:assIV} holds for $m=0$ we have $\tilde a(x) + \alpha \in R(p,\mathcal{D})$ as well and, hence, \eqref{eq:Aontheway} implies that $a(x) + \alpha x^{m} \in R(p,\mathcal{D})$.
\end{proof}

After this preparation we turn to the proof of Theorem~\ref{th:1}.

\begin{proof}[Proof of Theorem~\ref{th:1}]
Our goal is to apply Lemma~\ref{lem:inductive}. To this end let $a \in R(p,\mathcal{D})$ and $\alpha \in \Delta$ be given. We have to show that $a(x) + \alpha \in R(p,\mathcal{D})$.

Since $a \in R(p,\mathcal{D})$ we may write $a(x)\equiv \sum_{j=0}^{\ell-1}d_j x^j\pmod{p}$ with $d_0,\ldots,d_{\ell-1}\in\mathcal{D}$. For convenience, in what follows we set $d_j =0$ for $j \ge \ell$, $p_n=1$, and $p_j=0$ for $j > n$. Then $a(x)\equiv \sum_{j=0}^{\infty}d_j x^j \pmod{p}$. Since $\alpha + d_0 \in Z + \mathcal{D}$ (note that $\Delta \subset Z$), condition (i) implies that there is $\delta_0 \in \Delta$ and $b_0 \in \mathcal{D}$ such that $\alpha + d_0 = b_0 - \delta_0 p_0$. Adding $\delta_0p(x)$ to $a(x)+\alpha$ thus yields
\begin{equation}\label{eq:startt0}
a(x)+\alpha \equiv b_0 + \sum_{j=1}^{\infty} (d_j + \delta_0 p_j) x^j \pmod{p}.
\end{equation}
We want to prove that for each $t \ge 0$ the sum $a(x)+\alpha$ can be written in the form
\begin{equation}\label{eq:tbel}
a(x)+\alpha \equiv \sum_{j=0}^t b_j x^j + \sum_{j=t+1}^\infty (d_j + \delta_0p_j + \delta_1p_{j-1} + \cdots + \delta_tp_{j-t})x^j\bmod{p}
\end{equation}
with $b_j\in \mathcal{D}$ and $\delta_j\in \Delta$ for $0\le j \le t$. Indeed, we prove this by induction. Since this is true for $t=0$ by \eqref{eq:startt0} assume that it is true for some given value $t \ge 0$. The coefficient of $x^{t+1}$ in \eqref{eq:tbel} is $d_{t+1} + s$ with
\[
s=  \delta_0p_{t+1} + \delta_1p_{t} + \cdots + \delta_tp_{1}.
\]
As $p_j=0$ for $j>n$ the sum $s$ has at most $n$ nonzero summands each of which is of the form $\delta_jp_{t+1-j}$ with $\delta_j\in \Delta$ and $t-n+1\le j \le t$. Thus $s\in Z$ and, hence, $d_{t+1} + s \in \mathcal{D} + Z$. Now by condition (i) there exists $b_{t+1} \in \mathcal{D}$ and $\delta_{t+1} \in \Delta$ such that
\[
d_{t+1} + s = b_{t+1} - \delta_{t+1} p_0.
\]
Thus, adding $\delta_{t+1}p(x)x^{t+1}$  to \eqref{eq:tbel} we obtain a similar expression for $a(x)+\alpha$ with $t$ replaced by $t+1$. Thus, by induction, \eqref{eq:tbel} holds for all $t\ge 0$. Note that the sum in \eqref{eq:tbel} is finite since $p_j=0$ for $j>n$.

Assume now that $t \ge \ell-1$ in \eqref{eq:tbel}. Then for $j \ge t+1$ we have $d_j=0$ and, hence, the coefficient of $x^j$ has the form $\delta_0p_{j} + \delta_1p_{j-1} + \cdots + \delta_tp_{j-t} \in Z$. By (ii) this implies that $\delta_0p_{j} + \delta_1p_{j-1} + \cdots + \delta_tp_{j-t} \in \mathcal{D} \cup (\mathcal{D}  - p_0)$. This entails that $\delta_j\in \{0,1\}$ for $j \ge t+1$. Hence, if $t \ge \ell-1 + n$ for each of the nonzero summands of $\delta_0p_{j} + \delta_1p_{j-1} + \cdots + \delta_tp_{j-t}$ the coefficient $\delta_i$ equals $1$ and thus the sum belongs to $\mathcal{D}$ by (iii). Consequently, in the representation \eqref{eq:tbel} for $t \ge \ell - 1 + n$ all the coefficients belong to $\mathcal{D}$ and, since this sum is finite, $a(x)+ \alpha \in R(p,\mathcal{D})$.  Thus the condition of Lemma~\ref{lem:inductive} is satisfied and we may apply the lemma to conclude that $(p,\mathcal{D})$ is an GNS with finiteness property. This proves the theorem.
\end{proof}

\section{The finiteness property for large constant terms}\label{sec:largegns}

One of the main results of this paper is a generalization of a result of B.~Kov\'acs~\cite[Section~3]{Kovacs} that will be stated and proved in the present section. We begin with some notation. We denote by $\mathbf{e}_1=(1,0,\ldots,0)\in \mathbb{R}^k$ the first canonical basis vector of $\mathbb{R}^k$. Let $M\subset \mathbb{R}^k$. For $\varepsilon > 0$ we set
\[
(M)_\varepsilon := \{\mathbf{x} \in \mathbb{R}^k\;:\; ||\mathbf{x}-\mathbf{y}||_\infty < \varepsilon \hbox{ for some } \mathbf{y}\in M \}
\]
for the \emph{$\varepsilon$}-neighborhood of a set $M$. Moreover, ${\rm int}_+$ is the interior taken w.r.t.\ the subspace topology on $\{(r_1,\ldots, r_k) \in \mathbb{R}^k\;:\; r_1 \ge 0\}$. The symbol ${\rm int}_-$ is defined by replacing $r_1 \ge 0$ with $r_1 \le 0$.


\begin{theorem}\label{th:newKovacs}
Let $\mathbb{K}$ be a number field of degree $k$ and let $\mathcal{O}$ be an order in $\mathbb{K}$. Let a monic polynomial $p\in \mathcal{O}[x]$ and a bounded fundamental domain $\mathcal{F}$ for the action of $\mathbb{Z}^k$ on $\mathbb{R}^k$ be given. Suppose that
\begin{itemize}
\item $\mathbf{0} \in {\rm int}(\mathcal{F} \cup (\mathcal{F}-\mathbf{e}_1))$ and
\item $\mathbf{0} \in {\rm int}_+(\mathcal{F})$.
\end{itemize}
Then there is $\eta > 0$ such that $(p(x + \alpha),D_\mathcal{F})$ has the finiteness property whenever $\alpha = m_1\omega_1 + \dots + m_k \omega_k\in \mathcal{O}$ satisfies $\max\{1,|m_2|,\ldots,|m_k|\}< \eta m_1 $.
\end{theorem}

\begin{rem}
Note that this implies that for each bounded fundamental domain $\mathcal{F}$ satisfying \begin{itemize}
\item $\mathbf{0} \in {\rm int}(\mathcal{F} \cup (\mathcal{F}-\mathbf{e}_1))$ and
\item $\mathbf{0} \in {\rm int}_+(\mathcal{F})$
\end{itemize}
the family $\mathcal{G}_\mathcal{F}$ of GNS contains infinitely many GNS with finiteness property.
\end{rem}

\begin{proof}
Our goal is to apply Theorem~\ref{th:1}.

Choose $\varepsilon_1>0$ in a way that the $\varepsilon_1$-ball around $\mathbf{0}$ in $\R^k$ w.r.t.\ the norm $||\cdot||_\infty$ is contained in ${\rm int}(\mathcal{F} \cup (\mathcal{F}-\mathbf{e}_1))$. Since the union $\mathcal{F} + \mathbb{Z}^k$ is a locally finite union of bounded sets, the definition of the neighbor set $\mathcal{N}$ implies that there exists $\varepsilon_2 >0$ such that $(\mathcal{F})_{\varepsilon_2} \cap (\mathcal{F} + \mathbf{z}) = \emptyset$ for each $\mathbf{z} \in\mathbb{Z}^k\setminus\mathcal{N}$. Let now 
\begin{equation}\label{eq:eeppss}
\varepsilon = \min\{\varepsilon_1,\varepsilon_2\}. 
\end{equation}

We write $p(x+\alpha)=x^n+p_{n-1}(\alpha)x^{n-1}+\cdots+p_0(\alpha)$. Then there exist polynomials $q_j\in \mathbb{Z}[x]$ such that
\[
p(x+\alpha) = \sum_{j=1}^k q_j(x+\alpha)\omega_j = \sum_{j=1}^k \left(
\delta_{j1}x^n+p_{j,n-1}(\alpha)x^{n-1}+\cdots+p_{j0}(\alpha)
\right)\omega_j
\]
with $p_{jl}(\alpha)\in \mathbb{Z}$ and $\delta_{ij}$ being the Kronecker symbol. It is easy to see from the definition of these coefficients (see also~\cite[p.~294]{KovacsPetho}) that $p_{10}(\alpha)$ grows faster than all the other coefficients if $\eta\to 0$, more precisely, we have
\begin{equation}\label{eq:pc1}
p_{jl}(\alpha) \ll \eta p_{10}(\alpha), \qquad (j,l)\not=(1,0),\quad 1\le j\le k, \quad 0\le l < n
\end{equation}
for $\eta\to0$ (note that $\eta\to 0$ entails that $m_1 \to \infty$). Moreover, we see that
\begin{equation}\label{eq:pc3}
p_{jl}(\alpha) \ll \eta p_{1l}(\alpha), \qquad \quad 2\le j\le k, \quad 0\le l < n,
\end{equation}
for $\eta\to0$ and,
\begin{equation}\label{eq:pc2}
p_{1l}(\alpha) \ge 0 \qquad \hbox{for }0\le l < n \hbox{ and } \eta \hbox{ small (and, hence, $m_1$ large) enough}.
\end{equation}

Let now $\zeta \in Z=Z(\alpha)$ be given\footnote{By the notation $Z(\alpha)$ we indicate that the set $Z$ depends on $\alpha$ since the coefficients $p_j(\alpha)$ are functions in $\alpha$.}. Then by the definition of $Z$ the estimates in \eqref{eq:pc1} imply
that $\zeta=\zeta_1\omega_1+\cdots + \zeta_k\omega_k$ with
\begin{equation}\label{eq:zc:1}
\zeta_{j} \ll \eta p_{10}(\alpha), \quad 1\le j\le k,
\end{equation}
for $\eta\to0$.

We now show that 
\[
\zeta + D_{\mathcal{F},p_0(\alpha)} \subset \bigcup_{\delta \in \Delta} (D_{\mathcal{F},p_0(\alpha)} + p_0(\alpha)\delta) \quad\hbox{and}\quad\zeta \in D_{\mathcal{F},p_0(\alpha)} \cup (D_{\mathcal{F},p_0(\alpha)} - p_0(\alpha))
\]
holds for small $\eta$. Note first that there exists $\mathbf{r}=(r_1,\ldots, r_k) \in \mathbb{Q}^k$ with
\begin{equation}\label{eq:zetaquot}
\frac{\zeta}{p_0(\alpha)} = r_1\omega_1+ \cdots + r_k \omega_k.
\end{equation}
Since $p_0(\alpha) = \sum_{j=1}^k p_{j0}(\alpha)\omega_j$ this implies that
\[
\zeta_1\omega_1+\cdots + \zeta_k\omega_k = (r_1\omega_1+ \cdots + r_k \omega_k)(p_{10}(\alpha)\omega_1+ \cdots + p_{k0}(\alpha)\omega_k).
\]
Now multiplying the brackets on the right hand side and observing that $\omega_1=1$ the estimate in \eqref{eq:pc1} yields, setting $r=\max\{|r_1|,\ldots, |r_k|\}$,
\[
\zeta_1\omega_1+\cdots + \zeta_k\omega_k = p_{10}(\alpha)(r_1+O(\eta r)) \omega_1+ \cdots + p_{10}(\alpha)(r_k +O(\eta r))\omega_k.
\]
Let $j_0$ be an index with $r=|r_{j_0}|$. For this index we have
\[
\zeta_{j_0} =  p_{10}(\alpha)r_{j_0}(1+O(\eta)).
\]
Using \eqref{eq:zc:1} this implies that $r=|r_{j_0}|$ tends to zero for $\eta \to 0$. Thus for $\eta$ small enough we have $r < \varepsilon$ with $\varepsilon$ as in \eqref{eq:eeppss} and, hence, $||\mathbf{r}||_\infty=||(r_1,\ldots,r_k)||_\infty < \varepsilon$. By the choice of $\varepsilon$ this implies that 
\begin{equation}\label{eq:rel1}
\mathbf{r} + \mathcal{F} \subset \bigcup_{\mathbf{n}\in \mathcal{N}} (\mathcal{F} + \mathbf{n}) \quad\hbox{and}\quad \mathbf{r} \in \mathcal{F} \cup (\mathcal{F} - \mathbf{e}_1)
\end{equation}
hold for $\eta$ small enough. Multiplying both relations in \eqref{eq:rel1} by $p_0(\alpha) \cdot \boldsymbol{\omega}$ this yields by \eqref{eq:zetaquot} and the definition of $\Delta$ in \eqref{eq:deltaz} that 
\begin{equation}\label{eq:rel2}
\begin{split}
\zeta + p_0(\alpha)\cdot(\mathcal{F}\cdot \boldsymbol{\omega}) &\subset \bigcup_{\delta\in \Delta} p_0(\alpha)\cdot(\mathcal{F}\cdot \boldsymbol{\omega}) + p_0(\alpha)\delta \quad\hbox{and}\\ \zeta &\in p_0(\alpha)\cdot(\mathcal{F}\cdot \boldsymbol{\omega}) \cup (p_0(\alpha)\cdot(\mathcal{F}\cdot \boldsymbol{\omega}) - p_0(\alpha))
\end{split}
\end{equation} 
hold for $\eta$ small enough. Intersecting the relations in \eqref{eq:rel2} with $\mathcal{O}$, and using the definition of $D_{\mathcal{F},p_0(\alpha)}$ in \eqref{eq:DF} this implies that 
\[
\zeta + D_{\mathcal{F},p_0(\alpha)} \subset \bigcup_{\delta \in \Delta} (D_{\mathcal{F},p_0(\alpha)} + p_0(\alpha)\delta) \quad\hbox{and}\quad\zeta \in D_{\mathcal{F},p_0(\alpha)} \cup (D_{\mathcal{F},p_0(\alpha)} - p_0(\alpha))
\]
hold for $\eta$ small enough. Since $\zeta \in Z$ was arbitrary we have shown that there is $\eta_1 > 0$ with
\begin{equation}\label{eq:pf0}
Z+D_{\mathcal{F},p_0(\alpha)} \subset D_{\mathcal{F},p_0(\alpha)} + p_0(\alpha)\Delta \qquad \hbox{for } \eta < \eta_1
\end{equation}
and
\begin{equation}\label{eq:pf1}
Z \subset D_{\mathcal{F},p_0(\alpha)} \cup (D_{\mathcal{F},p_0(\alpha)}-p_0(\alpha)) \qquad \hbox{for }\eta < \eta_1.
\end{equation}
This implies conditions (i) and (ii) of Theorem~\ref{th:1}.

If we choose $\zeta'= \sum_{j \in J}p_j(\alpha)$ for some $J\subseteq \{1,\ldots, n\}$, there exist $r'_1,\ldots, r'_k \in \mathbb{Q}$ with
\[
\frac{\zeta'}{p_0(\alpha)} = r'_1\omega_1+ \cdots + r'_k \omega_k.
\]
and, hence, writing $\zeta' = \zeta_1'\omega_1 + \cdots+ \zeta_k'\omega_k$, we get
\begin{equation}\label{eq:zetaprimedef}
\zeta'_1\omega_1+\cdots + \zeta'_k\omega_k = (r_1'\omega_1+ \cdots + r'_k \omega_k)(p_{10}(\alpha)\omega_1+ \cdots + p_{k0}(\alpha)\omega_k).
\end{equation}
Then by the same arguments as above we derive that
\begin{equation}\label{eq:naz}
||(r_1',\ldots,r_k')||_\infty < \varepsilon \qquad \hbox{with $\varepsilon$ as in \eqref{eq:eeppss} and $\eta$ small enough.}
\end{equation}
Observe that by \eqref{eq:pc3} and \eqref{eq:pc2} we have
\begin{equation}\label{eq:nnzeta}
\zeta_j' \ll \eta \zeta_1' \qquad (2 \le j \le k)
\end{equation}
and $\zeta_1' >0$ for $\eta\to 0$. Let $j_0$ be an index with $|r'_{j_0}|=\max\{|r_1'|,\ldots,|r_k'|\}$ and assume that $j_0 \ge 2$. Then by \eqref{eq:zetaprimedef} and \eqref{eq:pc1}
\[
\zeta_{j_0}' = p_{10}(\alpha)r'_{j_0}(1+ O(\eta)) \gg p_{10}(\alpha)r'_{1} + O(p_{10}(\alpha) \eta r' )) = \zeta_1',
\]
a contradiction to $\eqref{eq:nnzeta}$. Thus $j_0=1$ and
\[
\zeta_{1}' = p_{10}(\alpha)r'_{1}(1+ O(\eta))
\]
and by \eqref{eq:pc2} we conclude that

\begin{equation}\label{eq:nbz}
r_1' >0 \qquad \hbox{for $\eta$ small enough.}
\end{equation}
Thus, by \eqref{eq:naz} and \eqref{eq:nbz} and a similar reasoning as above there is $\eta_2 > 0$ with
\begin{equation}\label{eq:pf2}
\bigg\{ \sum_{j\in J} p_j(\alpha) \;:\; J\subseteq\{1,\ldots, n\}  \bigg\} \subseteq D_{\mathcal{F},p_0(\alpha)} \qquad\hbox{for } \eta < \eta_2.
\end{equation}
This shows that also condition (iii) of Theorem~\ref{th:1} is stisfied.

Summing up we see that by \eqref{eq:pf0}, \eqref{eq:pf1}, and \eqref{eq:pf2} the result follows from Theorem~\ref{th:1} with $\eta=\min\{\eta_1,\eta_2\}$.
\end{proof}

Theorem~\ref{th:newKovacs} immediately admits the following corollary.

\begin{cor}\label{cor:newKovacs}
Let $\mathbb{K}$ be a number field of degree $k$ and let $\mathcal{O}$ be an order in $\mathbb{K}$. Let a monic polynomial $p\in \mathcal{O}[x]$ and a bounded fundamental domain $\mathcal{F}$ for the action of $\mathbb{Z}^k$ on $\mathbb{R}^k$ be given. If $\mathbf{0} \in {\rm int}(\mathcal{F})$ then there is $\eta > 0$ such that $(p(x + \alpha),D_\mathcal{F})$ has the finiteness property whenever $\alpha = m_1\omega_1 + \dots + m_k \omega_k\in \mathcal{O}$ satisfies $\max\{1,|m_2|,\ldots,|m_k|\}< \eta |m_1|$.
\end{cor}

\begin{proof}
Again we want to apply Theorem~\ref{th:1}.
Choose $\varepsilon_1>0$ in a way that the $\varepsilon_1$-ball around $\mathbf{0}$ w.r.t.\ $||\cdot||_\infty$ is contained in ${\rm int}(\mathcal{F})$. Since the union $\mathcal{F} + \mathbb{Z}^k$ is a locally finite union of bounded sets, the definition of the neighbor set $\mathcal{N}$ implies that there exists $\varepsilon_2 >0$ such that $(\mathcal{F})_{\varepsilon_2} \cap (\mathcal{F} + \mathbf{z}) = \emptyset$ for each $\mathbf{z} \not\in\mathcal{N}$. Let now $\varepsilon = \min\{\varepsilon_1,\varepsilon_2\}$.

Define
\[
\Delta' = (\mathcal{N}\cup\{\mathbf{e}_1\})\cdot\boldsymbol{\omega}
\quad
\hbox{and}
\quad
Z'=\bigg\{
\sum_{j=1}^n \delta_j p_j(\alpha) \;:\; \delta_j\in \Delta'
\bigg\}.
\]
Let now $\zeta \in Z'=Z'(\alpha)$ be given. In the same way as we showed \eqref{eq:pf0} and \eqref{eq:pf1} in the proof of Theorem~\ref{th:newKovacs} we can show, using $\varepsilon$ as defined above, that 
\[
\zeta + D_{\mathcal{F},p_0(\alpha)} \subset \bigcup_{\delta \in \Delta} (D_{\mathcal{F},p_0(\alpha)} + p_0(\alpha)\delta)\quad\hbox{and}\quad \zeta \in D_{\mathcal{F},p_0(\alpha)} 
\]
hold for small $\eta$. Thus, since $\zeta \in Z'$ was arbitrary and $Z \subset Z'$, there is $\eta_1 > 0$ with
\begin{equation*}\label{eq:pf0A}
Z+D_{\mathcal{F},p_0(\alpha)} \subset D_{\mathcal{F},p_0(\alpha)} + p_0(\alpha)\Delta \qquad \hbox{for } \eta < \eta_1,
\end{equation*}
\begin{equation*}\label{eq:pf1A}
Z \subset D_{\mathcal{F},p_0(\alpha)} \qquad \hbox{for }\eta < \eta_1.
\end{equation*}
Moreover, because $\left\{ \sum_{j\in J} p_j(\alpha) \;:\; J\subseteq\{1,\ldots, n\}  \right\} \subset Z'$,
\begin{equation*}
\bigg\{ \sum_{j\in J} p_j(\alpha) \;:\; J\subseteq\{1,\ldots, n\}  \bigg\} \subseteq D_{\mathcal{F},p_0(\alpha)} \qquad\hbox{for } \eta < \eta_1.
\end{equation*}
This implies conditions (i), (ii), and (iii) of Theorem~\ref{th:1} and we are done.
\end{proof}

\begin{rem} \label{rem:44}
Under the conditions of Theorem~\ref{th:newKovacs}
\[
\exists \; M\in\mathbb{N}:\;(p(x+m),\mathcal{F}) \hbox{ is a GNS with finiteness property for } m\ge M,
\]
while under the more restrictive conditions of Corollary~\ref{cor:newKovacs}
\[
\exists\; M\in\mathbb{N}:\; (p(x\pm m),\mathcal{F})  \hbox{ is a GNS with finiteness property for }\; m\ge M.
\]
\end{rem}

\begin{rem} \label{rem:45}
Looking back at Example~\ref{ex:ex} we see that canonical number systems and number systems corresponding to the sail satisfy the conditions of Theorem~\ref{th:newKovacs}. Symmetric CNS even satisfy the conditions of Corollary~\ref{cor:newKovacs}. The number systems corresponding to the square $\mathcal{F}=[0,1)^2$ do not fit in our framework. In this case, $\mathcal{F}$ needs to be translated appropriately in order to make our results applicable.
\end{rem}

We will see in the next section that $(p(x-m),\mathcal{F})$ doesn't have the finiteness property for large $m$ under the conditions of Theorem~\ref{th:newKovacs}.

Before this we deal with the following conjecture of Akiyama, see Brunotte \cite{Brunotte2017}: let $p\in \Z[x]$ be a CNS polynomial. Then there exists $M$ such that $p(x)+m$ is a CNS polynomial for all $m\ge M$. Theorem~\ref{th:newKovacs} implies results concerning this conjecture even for polynomials over orders.

\begin{cor} \label{cor:Akiyama}
Let $\mathbb{K}$ be a number field of degree $k$ and let $\mathcal{O}$ be an order in $\mathbb{K}$. Let a monic polynomial $p\in \mathcal{O}[x]$ and a bounded fundamental domain $\mathcal{F}$ for the action of $\mathbb{Z}^k$ on $\mathbb{R}^k$ be given. Suppose that $\mathbf{0} \in {\rm int}(\mathcal{F})$ then there is $\eta > 0$ such that $(p(x) \pm \alpha,D_\mathcal{F})$ has the finiteness property whenever $\alpha = m_1\omega_1 + \dots + m_k \omega_k\in \mathcal{O}$ satisfies $\max\{1,|m_2|,\ldots,|m_k|\}< \eta |m_1|$.
\end{cor}

\begin{proof}
  Repeat the proof of Corollary \ref{cor:newKovacs} with $p(x)\pm \alpha$ instead of $p(x\pm \alpha)$.
\end{proof}

\begin{rem}
  If $k=1$, and $0< \varepsilon<1$ then $\mathcal{F}_{\varepsilon}=[-\varepsilon,1-\varepsilon)$ satisfies the conditions of Corollary \ref{cor:Akiyama}, hence for any $p\in \Z[x]$ there exists $M\in \Z$ depending only on $\varepsilon$ and the size of the coefficients of $p$ such that $(p(x)\pm m, \mathcal{F}_{\varepsilon})$ is a GNS with finiteness property in $\Z[x]$.\\
The assumptions of Theorem \ref{th:newKovacs} hold for $\mathcal{F}_{\varepsilon}$ even if $ \varepsilon=0$. Hence, if all coefficients of $p$ are non-negative, then \eqref{eq:pc2} holds and we can conclude $r_1'\ge 0$ as in the proof of Theorem \ref{th:newKovacs}. Hence in this case $(p(x)+m, \mathcal{F}_{0})$ is a GNS with finiteness property in $\Z[x]$.\\
  However, if some of the coefficients of $p$ are negative, then \eqref{eq:pc2} and $r_1'\ge 0$ fails and, hence, we do not have similar statement. The example $p=x^2-2x+2$ shows that $(p(x)+m, \mathcal{F}_{0})$ is not a GNS with finiteness property in $\Z[x]$ for any $m\ge 0$.

\end{rem}
\begin{rem}
  If there are infinitely many units in $\mathcal{O}$ then for all $p\in \mathcal{O}[x]$ there exist infinitely many $\alpha \in \mathcal{O}$ such that the constant term of $p(x) + \alpha$, {\it i.e.}, $p(0) + \alpha$ is a unit, hence $p(x) + \alpha$ is not GNS with finiteness property. Notice that Condition (iii) of Theorem \ref{th:1} holds under the assumptions of Corollary \ref{cor:Akiyama} only if the norm of $p(0) + \alpha$ is large.
\end{rem}

\section{GNS without finiteness property}

The main result of this section complements the results of Section~\ref{sec:largegns}. We start with a partial generalization of \cite[Theorem~3]{KovacsPetho} to polynomials with coefficients of $\mathcal O$ that will be needed in its proof.

\begin{lemma} \label{non-ECNS2}
Let $(p,\mathcal{D})$ be a GNS. If there exist $h \in \mathbb{N}$, $d_0,d_1,\dots,d_{h-1} \in \mathcal{D}$ not all equal to $0$ and $q_1,q_2 \in \mathcal{O}[x]$ with
    \begin{equation} \label{non-ECNSeq1}
   \sum_{j=0}^{h-1} d_j x^j  =  (x^{h} - 1) q_1(x) + q_2(x)p(x).
    \end{equation}
then $(p,\mathcal{D})$ doesn't have the finiteness property.
  \end{lemma}

\begin{proof}
Assume that \eqref{non-ECNSeq1} holds for some $h\in\mathbb{N}$, $d_0,d_1,\dots,d_{h-1} \in \mathcal{D}$ not all equal to zero and $q_1(x), q_2(x) \in \mathcal{O}[x]$. This implies that
$$
-q_1(x) \equiv \sum_{j=0}^{h-1} d_j x^j + x^{h} (-q_1(x)) \equiv \sum_{k=0}^{\ell-1} \sum_{j=0}^{h-1} d_j x^{kh+j} + x^{\ell h} (-q_1(x)) \pmod{p}
$$
holds for all $\ell \in\mathbb{N}$. Since $\mathcal{D}$ is a complete residue system modulo $p(0)$ this implies that a possible finite digit representation
$$
-q_1(x) \equiv \sum_{j =0}^{L-1} b_j x^j \pmod{p}
$$
must satisfy $L \ge h\ell$ for all $\ell\in\mathbb{N}$. Thus $L$ cannot be finite, a contradiction. This implies that $(p,\mathcal{D})$ does not have the finiteness property.
\end{proof}

Our main result in this section is the following theorem.

\begin{theorem} \label{non-ECNSmain}
Let $\mathbb{K}$ be a number field of degree $k$ and let $\mathcal{O}$ be an order in $\mathbb{K}$. Let a monic polynomial $p\in \mathcal{O}[x]$ and a bounded fundamental domain $\mathcal{F}$ for the action of $\mathbb{Z}^k$ on $\mathbb{R}^k$ containing $\mathbf{0}$ be given. Suppose that
$\mathbf{0}\in{\rm int}_{-}(\overline{\mathcal{F}}-\mathbf{e}_1)$.
There exists $M \in \mathbb{N}$ such that $(p(x-m),D_\mathcal{F})$ doesn't have the finiteness property for $m\ge M$.
\end{theorem}

\begin{proof}
For an integer $m$ set $\Pi_m(x) = p(x-m)$. In the sequel we examine the constant term of $\Pi_m(x)$, which is $\Pi_m(0) = p(-m)$. We claim that if $m$ is large enough then $\Pi_m(0)\in D_{\mathcal{F},p(-m-1)}$.

Assume that our claim is true. Performing Euclidean division of $\Pi_{m+1}(x)$ by $(x-1)$ we obtain a polynomial $s_{m+1}(x) \in \mathcal{O}[x]$ such that
$
\Pi_{m+1} (x) = (x-1) s_{m+1}(x) + \Pi_{m+1}(1).
$
As $\Pi_{m+1}(1) = p(-m)$ the last identity is equivalent to
$$
p(-m) = (x-1) (-s_{m+1}(x)) + \Pi_{m+1}(x).
$$
By the claim $p(-m)$ belongs to the digit set $D_{\mathcal{F},p(-m-1)}$ if $m$ is large enough. Applying Lemma~\ref{non-ECNS2} with $h=1, d_0=p(-m), q_1(x) = -s_{m+1}(x), q_2(x) =1$, $p(x) = \Pi_{m+1}(x)$ and $\mathcal{D}=D_{\mathcal{F},p(-m-1)}$ we conclude that $(\Pi_{m+1},D_{\mathcal{F},p(-m-1)})$ is not a GNS with finiteness property whenever $m$ is large enough.

It remains to prove the claim. Let $p(x) = x^n + p_{n-1} x^{n-1}+\dots+p_0$. Then
\begin{align*}
  \frac{\Pi_m^{(i)}(0)}{\Pi^{(i)}_{m+1}(0)}&-1 = \frac{p(-m)^{(i)}}{p(-m-1)^{(i)}}-1 \\
   &= \frac{(-m)^n + p_{n-1}^{(i)}(-m)^{n-1}- (-m-1)^n-p_{n-1}^{(i)}(-m-1)^{n-1}+ O(m^{n-2})}{ (-m-1)^n+p_{n-1}^{(i)}(-m-1)^{n-1}+ O(m^{n-2})}\\
  &=\frac{-(-1)^n n m^{n-1}+O(m^{n-2})}{(-m)^n + O(m^{n-1})}\\
  &= -\frac{n}{m}+O(m^{-2}),
\end{align*}
hence,
\begin{equation}\label{eq:pipi0}
\frac{\Pi_m^{(i)}(0)}{\Pi^{(i)}_{m+1}(0)} = 1- \frac{n}{m}+O(m^{-2})
\end{equation}
for $i=1,\dots,k$.

Setting
\begin{equation}\label{eq:new5.3}
\frac{\Pi_m(0)}{\Pi_{m+1}(0)} = \sum_{j=1}^{k} r_{mj} \omega_j,
\end{equation}
by the definition of $D_{\mathcal{F},p(-m-1)}=D_{\mathcal{F},\Pi_{m+1}(0)}$ our claim is proved if we show that $(r_{m1},\ldots,r_{mk})\in \mathcal{F}$ holds for large $m$. (Note that, as $\Pi_m(0)/\Pi_{m+1}(0)$ belongs to $\K$, we have $r_{m1},\dots,r_{mk}\in \Q$.)

Taking conjugates, equation \eqref{eq:new5.3} implies
\begin{equation}\label{eq:pipi}
\frac{\Pi_m^{(i)}(0)}{\Pi^{(i)}_{m+1}(0)} = \sum_{j=1}^{k} r_{mj} \omega_j^{(i)}, \;i=1,\dots,k.
\end{equation}
This is a system of linear equations in the unknowns $r_{mj}, j=1,\dots,k$ with coefficient matrix $(\omega_j^{(i)})_{i,j=1,\dots,k}$. As $\omega_1,\dots,\omega_k$ is a basis of $\mathcal{O}$ the determinant of $(\omega_j^{(i)})$ is not zero. Moreover, as $\omega_1=1$ the first column of $(\omega_j^{(i)})$ is $\bf 1$, the vector which consists only of ones.

We estimate the solutions of \eqref{eq:pipi} by using Cramer's rule. If $j>1$ then to get the matrix of the numerator  we have to replace the $j$-th column of $(\omega_j^{(i)})$ by the vector $\left(\frac{\Pi_m^{(1)}(0)}{\Pi^{(1)}_{m+1}(0)}, \dots, \frac{\Pi_m^{(k)}(0)}{\Pi^{(k)}_{m+1}(0)}\right)={\bf 1} (1- \frac{n}{m} +O(m^{-2}))$. As the first column of $(\omega_j^{(i)})$ was not altered, {\it i.e.}, it is $\bf 1$, the determinant of this matrix is $O(m^{-1})$. As the determinant of the denominator matrix is constant, {\it i.e.}, $\det(\omega_j^{(i)})$, we get
\begin{equation}\label{eq:r2}
 r_{mj} = O(m^{-1}),\; j=2,\dots,k.
\end{equation}
If $j=1$ then to get the matrix of the numerator  we have to replace the first column of $(\omega_j^{(i)})$ by the vector $\left(\frac{\Pi_m^{(1)}(0)}{\Pi^{(1)}_{m+1}(0)}, \dots, \frac{\Pi_m^{(k)}(0)}{\Pi^{(k)}_{m+1}(0)}\right) ={\bf 1} (1- \frac{n}{m} +O(m^{-2}))$, thus
$$
r_{m1} = 1-\frac{n}{m} +O(m^{-2}).
$$

This yields that
\begin{equation}\label{eq:r3}
1- \frac{n}{2m} < r_{m1} < 1
\end{equation}
holds for $m$ large.
Thus, since $\mathbf{0}\in{\rm int}_{-}(\overline{\mathcal{F}}-\mathbf{e}_1)$ by assumption, \eqref{eq:r2}, and \eqref{eq:r3} impliy that $(r_{m1},\ldots,r_{mk})\in \mathcal{F}$ for large $m$ and, hence, $\Pi_m(0)\in D_{\mathcal{F},p(-m-1)}$  for large $m$, and the claim is proved.
\end{proof}

\section{GNS in number fields}\label{sec:6}

As mentioned in the introduction, the theory of generalized number systems started with investigations in the ring of integers of algebraic number fields. For an overview on related results we refer to Evertse and Gy\H{o}ry~\cite{Evertse_Gyory} and Brunotte, Huszti, and Peth\H{o} \cite{BHP}. To clarify the connection of our investigations with the earlier ones we need some definitions. Let $\mathbb{L}$ be a number field of degree $l$, and denote its ring of integers by $\mathcal{O}_{\mathbb{L}}$. Let $\alpha \in \mathcal{O}_{\mathbb{L}}$ and let $\mathcal{N}$ be a complete residue system modulo $\alpha$ containing $0$. The pair $(\alpha,\mathcal{N})$ is called a {\it number system in $\mathcal{O}_{\mathbb{L}}$}. If for each $\gamma \in \mathcal{O}_{\mathbb{L}}$ there exist integers $\ell\ge 0$, $d_0,\dots,d_{\ell-1}\in \mathcal{N}$ such that
$$
\gamma = \sum_{j=0}^{\ell-1}d_j \alpha^j
$$
then we say that $(\alpha,\mathcal{N})$ has the {\em finiteness property}.
If the digit set is chosen to be $\mathcal{N}=\mathcal{N}_0(\alpha)=\{0,1,\dots,|N_{\mathbb{L}/\mathbb{Q}}(\alpha)|-1\}$ then $(\alpha,\mathcal{N})$ is called a {\it canonical number system in $\mathcal{O}_{\mathbb{L}}$} (CNS for short). (For $p$ being an irreducible polynomial with root $\alpha$ and $\mathcal{O}_{\mathbb{L}}=\mathbb{Z}[\alpha]$, the finiteness property of $(\alpha,\mathcal{N})$ coincides with the finiteness property of $(p,\mathcal{N})$ according to the introduction.) Kov\'acs \cite{Kovacs} proved that there exists a canonical number system with finiteness property in $\mathcal{O}_{\mathbb{L}}$ if and only if $\mathcal{O}_{\mathbb{L}}$ admits a power integral bases. Later Kov\'acs and Peth\H{o} \cite{KovacsPetho} proved the stronger result.

\begin{proposition}[{Kov\'acs and Peth\H{o} \cite[Theorem~5]{KovacsPetho}}]
  Let $\mathcal{O}$ be an order in the algebraic number field $\mathbb{L}$. There exist $\alpha_1,\dots,\alpha_t \in \mathcal{O}$, $n_1,\dots,n_t \in \mathbb{Z}$, and $N_1,\dots,N_t$ finite subsets of $\mathbb{Z}$, which are all effectively computable, such that $(\alpha,\mathcal{N}_0(\alpha))$ is a canonical number system with finiteness property in $\mathcal{O}$ if and only if $\alpha=\alpha_i-h$ for some integers $i,h$ with $1\le i \le t$ and either $h\ge n_i$ or $h\in N_i$.
\end{proposition}

From Corollary \ref{cor:newKovacs} we derive that for number systems the relation is usually stronger, the theorem of Kov\'acs and Peth\H{o} describes a kind of ``boundary case'' {\it viz.} a case where $0\in\partial\mathcal{F}$.

\begin{theorem} \label{thm:newKovacsPetho}
Let $\mathbb{L}$ be a number field of degree $l$ and let $\mathcal{O}$ be an order in $\mathbb{L}$. Let $\mathcal{F}$ be a bounded fundamental domain for the action of $\mathbb{Z}$ on $\mathbb{R}$. If $0 \in {\rm int}(\mathcal{F})$ then all but finitely many generators of power integral bases of $\mathcal{O}$ form a basis for a number system with finiteness property. Moreover, the exceptions are effectively computable.
\end{theorem}

\begin{proof}
By Gy\H{o}ry \cite{Gyory} in $\mathcal{O}$ there exist up to translations by integers only finitely many generators of power integral bases, and they are effectively computable. Denote these finitely many generators by $\alpha_1,\dots,\alpha_t$ and denote the minimal polynomial of $\alpha_j$ by $p_j(x)$, $j=1,\dots,t$. Fix $1\le j \le t$. Note that $p_j(x)$ is monic and has rational integer coefficients. By Corollary~\ref{cor:newKovacs} (see especially Remark~\ref{rem:44}) there exists $M_j\in \Z$, such that $(p_j(x\pm m), \mathcal{F})$ is a GNS with finiteness property for all $m>M_j$. Fix such an $m$ and its sign $\delta$ too. Denote by $\mathcal{D}= D_{\mathcal{F},p_j(\delta m)}$ the digit set corresponding to $\mathcal{F}$ and $p_j(\delta m)$. Notice that $\mathcal{D}\subset \Z$ is a complete residue system modulo $p_j(\delta m)$.

Because of the finiteness property there exist for any $\gamma \in \mathcal{O}$ digits $d_0,\dots,d_{\ell-1} \in \mathcal{D}$ such that
$$
\gamma \equiv d_0+ d_1x+\dots +d_{\ell-1} x^{\ell-1} \pmod{p_j(x+\delta m)}.
$$
Inserting $\alpha_j - \delta m$ into this congruence and taking into consideration that $\alpha_j - \delta m$ is a zero of $p_j(x+\delta m)$ we obtain
$$
\gamma = d_0+ d_1(\alpha_j - \delta m)+\dots +d_{\ell-1} (\alpha_j - \delta m)^{\ell-1},
$$
hence, the pair $(\alpha_j - \delta m,\mathcal{D})$ is a number system with finiteness property in $\mathcal{O}$ provided that $m>M_j$.

If $1,\alpha,\dots,\alpha^{l-1}$ is a power integral basis of $\mathcal{O}$ then by Gy\H{o}ry's theorem there exist $1\le j \le t, \delta = \pm 1, m\in \N$ such that $\alpha = \alpha_j - \delta m$. We have seen in the last paragraph that all but finitely many $m$ the numbers $\alpha_j - \delta m$ together with the digit sets $\mathcal{D}$ form a number system with finiteness property.

Finally for each of the finitely many remaining values of $m$ one can decide algorithmically the finiteness property by Theorem~\ref{non-ECNS1}.
\end{proof}

\begin{rem}
Notice that the assumption $0\in {\rm int}(\mathcal{F})$ implies that $\{-1,0,1\} \subseteq D_{\mathcal{F},p_j(\delta m)}$ for all $m$ large enough. Of course $-1 \notin \mathcal{N}_0(\alpha+m)$, hence, the proof of Theorem~\ref{thm:newKovacsPetho} does not work in the case of canonical number systems.
\end{rem}

\begin{rem}
  Gy\H{o}ry's theorem holds for relative extensions as well. More precisely, if $\mathcal{O}$ is an order in an algebraic number field $\mathbb{K}$ and $p \in \mathcal{O}[x]$ is monic and irreducible then $\mathbb{L}= \mathcal{O}[x]/p\mathcal{O}[x]$ is a finite extension field of $\mathbb{K}$. Gy\H{o}ry \cite{Gyory} proved that if $\mathcal{U}$ is a ring and a free $\mathcal{O}$-module in $\mathbb{L}$, then it admits finitely many classes of $\mathcal{O}$-power integral bases. A representative of each class is effectively computable. Each class is closed under translation by elements of $\mathcal{O}$. To generalize Theorem~\ref{thm:newKovacsPetho} to this situation would require the generalization of Remark~\ref{rem:44} to all $m \in \mathcal{O}$, such that all conjugates of $m$ are large enough. We have no idea how to prove such a result.
\end{rem}

{\bf Acknowledgement} {\it We are very indebted to Professor K\'alm\'an Gy\H{o}ry. He suggested us to work out the GNS concept for orders of any number field, and not only for Euclidean number fields as we intended to do. He encouraged us continuously during our investigations. We also thank the anonymous referees for their thorough reading of the manuscript and their essential remarks, which made the quality of the presentation much better.}

%

\bibliographystyle{siam}
\bibliography{orderCNS}

\end{document}